\documentclass[10pt]{article}
\font\smallit=cmti10

\usepackage{amssymb,latexsym,amsmath,epsfig,amsthm}
\newtheorem{theorem}{\textbf{Theorem}}

\newtheorem{prop}{\textbf{Proposition}}

\theoremstyle{definition}

\makeatletter

\renewcommand\section{\@startsection {section}{1}{\z@}
{-30pt \@plus -1ex \@minus -.2ex}
{2.3ex \@plus.2ex}
{\normalfont\normalsize\bfseries}}

\renewcommand\subsection{\@startsection{subsection}{2}{\z@}
{-3.25ex\@plus -1ex \@minus -.2ex}
{1.5ex \@plus .2ex}
{\normalfont\normalsize\bfseries}}

\renewcommand{\@seccntformat}[1]{\csname the#1\endcsname. }

\makeatother

\begin{document}

\begin{center}
\uppercase{\bf Evaluation of Tachiya's algebraic infinite products involving Fibonacci and Lucas numbers}
\vskip 10pt
{\bf Jonathan Sondow
}\\
{\smallit 209 West 97th Street \#6F,
New York, NY 10025, USA\\
jsondow@alumni.princeton.edu}\\
\end{center}
\bigskip

\noindent{\bf Abstract.} In 2007, Tachiya gave necessary and sufficient conditions for the transcendence of certain infinite products involving Fibonacci numbers $F_k$ and Lucas numbers $L_k$. In the present note, we explicitly evaluate two classes of his algebraic examples. Special cases are
$$\prod_{n=1}^{\infty}\left(1+\frac{1}{F_{2^n+1}}\right)=\frac{3}{\varphi}, \qquad
\prod_{n=1}^{\infty}\left(1+\frac{1}{L_{2^n+1}}\right)=3-\varphi\, ,$$
where $\varphi=(1+\sqrt{5})/2$ is the golden ratio.\\



\pagestyle{myheadings} 
\thispagestyle{empty} 
\baselineskip=12.875pt 
\setcounter{page}{1}
\bigskip
\bigskip

%Section
\centerline{\bf INTRODUCTION}
\bigskip

For $k\ge0$, define the \emph{Fibonacci numbers} $F_k$ by $F_{k+2}=F_{k+1}+F_k$ with $F_0=0$ and $F_1=1$, and the \emph{Lucas numbers} $L_k$ by $L_{k+2}=L_{k+1}+L_k$ with $L_0=2$ and $L_1=1$.

In 2003, D.~Duverney and K.~Nishioka \cite{dn} used Mahler's method to give a transcendence criterion for general series. As applications, they established necessary and sufficient conditions for the transcendence of the numbers
$$\sum_{n=0}^{\infty} \frac{c_n}{F_{r^n}+d_n}, \qquad \sum_{n=0}^{\infty} \frac{c_n}{L_{r^n}+d_n},$$
where $r\ge2$ is an integer and $\{c_n\}_{n\ge0}$ and $\{d_n\}_{n\ge0}$ are suitable sequences of algebraic numbers.

By modifying their method, Y.~Tachiya \cite{yohei} in 2007 proved analogous results for infinite products. In particular, he gave necessary and sufficient conditions for the transcendence of the numbers
$$\prod_{n=0}^{\infty}\left(1+\frac{c_n}{F_{ar^n+b}}\right), \qquad \prod_{n=0}^{\infty}\left(1+\frac{c_n}{L_{ar^n+b}}\right),$$
where $r\ge2$, $a$, and $b$ are integers, $ar^n+b\neq0$ for $n\ge0$, and $\{c_n\}_{n\ge0}$ is a suitable sequence of algebraic numbers. In the special case where the sequence is constant, his Examples~1 and~2 can be stated as the following two theorems.

\begin{theorem}[{\bf Tachiya}]\label{yoheiF}
Let $a\ge1$, $b\ge0$, $r\ge2$, and $c$ be integers. Then
$$\prod_{n=0}^{\infty}\left(1+\frac{c}{F_{ar^n+b}}\right)$$
is algebraic if and only if

\noindent $\ (i)$ $c=0$, or

\noindent $(ii)$ $r=2$ and $c=F_b$.
\end{theorem}

In particular,
$$\prod_{n=0}^{\infty}\left(1+\frac{c}{F_{r^n}}\right)$$
is algebraic if and only if $c=0$.

\begin{theorem}[{\bf Tachiya}]\label{yoheiL}
Let $a\ge1$, $b\ge0,$ $r\ge2$, and $c$ be integers. Then
$$\prod_{n=0}^{\infty}\left(1+\frac{c}{L_{ar^n+b}}\right)$$
is algebraic if and only if

\noindent$\ \ (i)$ $c=0$, or

\noindent $\ (ii)$ $r=2$ and $c=L_b$, or

\noindent$(iii)$ $r=2,b=0$, and there exists a root of unity $\omega$ such that $\omega^{2^n}+\omega^{-2^n}=c$ for all large~$n$.
\end{theorem}

In particular, Tachiya \cite[p. 185]{yohei} obtained that, for any integers $c\neq0$ and $r\ge2$, the number
$$\prod_{\substack{n=1\\
 L_{r^n}\neq-c}}^{\infty}\left(1+\frac{c}{L_{r^n}}\right)$$
is transcendental, except for the two algebraic cases
$$\prod_{n=1}^{\infty}\left(1+\frac{-1}{L_{2^n}}\right)=\frac{\sqrt5}{4}, \qquad\prod_{n=1}^{\infty}\left(1+\frac{2}{L_{2^n}}\right)=\sqrt5,$$
which follow from the case $(iii)$ with $\omega=(1\pm\sqrt{-3})/2$ and $\omega=\pm1$, respectively. He says, ``These examples of algebraic infinite products involving Lucas numbers seem to be new.''

In the present note, we explicitly evaluate the algebraic infinite products in case~$(ii)$ of Theorems~\ref{yoheiF} and~\ref{yoheiL}. These formulas may also be new.

\newpage

%Section
\centerline{\bf THE FORMULAS}
\bigskip

\begin{prop} \label{PROP: F}
Let $F_k$ and $L_k$ be the $k$th Fibonacci and Lucas numbers, and let $a$ and~$b$ be positive integers. Then
\begin{equation} \label{EQ: Fab}
\prod_{n=1}^{\infty}\left(1+\frac{F_b}{F_{2^na+b}}\right)=\frac{1-(-1)^{b}\varphi^{-2a-2b}}{1-\varphi^{-2a}}
\end{equation}
and
\begin{equation} \label{EQ: Lab}
\prod_{n=1}^{\infty}\left(1+\frac{L_b}{L_{2^na+b}}\right)=\ \frac{1+(-1)^{b}\varphi^{-2a-2b}}{1-\varphi^{-2a}},
\end{equation}
where$$\varphi=\frac{1+\sqrt{5}}{2}$$ is the golden ratio. In particular,
\begin{equation} \label{EQ: F1}
\prod_{n=1}^{\infty}\left(1+\frac{1}{F_{2^n+1}}\right)=\frac{3}{\varphi}
\end{equation}
and
\begin{equation} \label{EQ: L1}
\prod_{n=1}^{\infty}\left(1+\frac{1}{L_{2^n+1}}\right)=3-\varphi.
\end{equation}
\end{prop}

\begin{proof}
(Compare \cite[p. 199]{yohei}.) From Binet's formula
$$F_k=\frac{\varphi^k-(-\varphi^{-1})^{k}}{\varphi+\varphi^{-1}}$$
we have
\begin{align*}
\prod_{n=1}^{\infty}\left(1+\frac{F_b}{F_{2^na+b}}\right)
=&\prod_{n=1}^{\infty}\left(1+\frac{\varphi^{b}-(-1)^{b}\varphi^{-b}}{\varphi^{2^na+b}-(-1)^{b}\varphi^{-2^na-b}}\right)\\
=&\prod_{n=1}^{\infty}\left(1+\frac{\varphi^{-2^na-b}\left(\varphi^b-(-1)^{b}\varphi^{-b}\right)}{1-(-1)^{b}\varphi^{-2^{n+1}a-2b}}\right)\\
=&\prod_{n=1}^{\infty} \left(1+\varphi^{-2^na}\right)\frac{1-(-1)^{b}\varphi^{-2^{n}a-2b}}{1-(-1)^{b}\varphi^{-2^{n+1}a-2b}}\\
=&\ \frac{1-(-1)^{b}\varphi^{-2a-2b}}{1-\varphi^{-2a}}\, ,
\end{align*}
by cancellation and the fact that
$$\prod_{n=1}^{\infty} \left(1+x^{2^n}\right)=\frac{1}{1-x^2} \qquad(\vert x\vert<1).$$
This proves the equality \eqref{EQ: Fab}. It and the relation $\varphi^2=\varphi+1$ yield
\begin{align*}
\prod_{n=1}^{\infty}\left(1+\frac{1}{F_{2^n+1}}\right)= \frac{1+\varphi^{-4}}{1-\varphi^{-2}}
=\frac{\varphi^{4}+1}{\varphi^{4}-\varphi^{2}}= \frac{3\varphi+3}{\varphi^2+\varphi}=\frac{3}{\varphi}\, ,
\end{align*}
proving \eqref{EQ: F1}. Using the formula $L_k=\varphi^k+(-\varphi)^{-k},$ the proofs of \eqref{EQ: Lab} and \eqref{EQ: L1} are similar and will be omitted.
\end{proof}

\end{document}